\documentclass[a4paper,10pt,twoside,openright]{article}
\usepackage{amssymb}
 \usepackage{amsthm}
\usepackage{amsmath,amssymb,amsopn,amsfonts,mathrsfs,amsbsy,amscd}

\newcommand{\prs}{\langle\;,\;\rangle}
\newcommand{\too}{\longrightarrow}
\newcommand{\om}{\omega}
\newcommand{\esp}{\quad\mbox{and}\quad}

\newcommand{\X}{{\cal X}}
\newcommand{\G}{{\mathfrak{g}}}

\newcommand{\h}{{\mathfrak{h}}}

\newcommand{\tr}{{\mathrm{tr}}}
\newcommand{\ric}{{\mathrm{ric}}}
\newcommand{\Ri}{{\mathrm{Ric}}}
\newcommand{\Ku}{{\mathrm{K}}}

\newcommand{\B}{{\cal B}}

\newcommand{\na}{\nabla}

\newcommand{\al}{\alpha}
\newcommand{\be}{\beta}
\newcommand{\we}{\wedge}

\newcommand{\la}{\lambda}

\font\bb=msbm10

\def\B{\hbox{\bb B}}
\def\R{\hbox{\bb R}}

\def\C{\hbox{\bb C}}
\newtheorem{Def}{Definition}[section]
\newtheorem{theo}{Theorem}[section]
\newtheorem{pr}{Proposition}[section]
\newtheorem{Le}{Lemma}[section]
\newtheorem{co}{Corollary}[section]

\newtheorem{rem}{Remark}

\title{\bf \textbf{ Four-dimensional homogeneous semi-symmetric neutral manifolds}}
\author{\textbf{A. Benroummane}\\ University   HASSAN 1. PB :218 Berrechid.\\Morocco \\abderrazzak.benroummane@uhp.ac.ma
 }
\begin{document}

\maketitle

\textbf{Abstract.} 

We  determine
all four-dimensional homogeneous semi-symmetric neutral manifolds.


\newpage

\section{Introduction}\label{section1}

A pseudo-Riemannian manifold $(M,g)$ is said to be semi-symmetric if its Riemannian curvature tensor $R$ satisfies $R.R=0$. This is equivalent to
\begin{equation}\label{eq1}
[R(X,Y),R(Z,T)]=R(R(X,Y)Z,T)+R(Z,R(X,Y)T),
\end{equation}
for any vector fields $X,Y,Z,T$.

 Semi-symmetric pseudo-Riemannian manifolds generalize obviously locally symmetric manifolds ($\na R=0$). They also generalize second-order locally symmetric manifolds ($\na^2R=0$ and $\na R\not=0$).  Semi-symmetric Riemannian manifolds have been first investigated by E. Cartan \cite{cartan} and the first example of a semi-symmetric not locally symmetric Riemannian manifold was given by Takagi \cite{takagi}. From 1983 to 1985, Szabo \cite{zabo, zabo1} gave a complete description of these manifolds. Recently, the author with M. Boucetta and A. Ikemakhen give a complet classification of 
 four-dimensional homogeneous semi-symmetric Lorentzian manifolds \cite{benromane}.
 More recently,  A. Haji-Badali and A. Zaeim give a complet classification of 
 four-dimensional   semi-symmetric neutral Lie groups 
\cite{ali}.

This paper is devoted to the study of semi-symmetric curvature algebraic tensors on the   vector space with metric of signature $(2,n)$ such that $n\geq2$ and a classification of 4-dimensional simply-connected semi-symmetric  homogeneous neutral manifolds which are not locally symmetric. There are our main results:
\begin{enumerate}
	\item Let $(V,\prs)$ be a vector space  with metric of signature $(2,n)$ and $\Ku:V\wedge V\too V\wedge V$ a semi-symmetric algebraic curvature tensor, i.e., $\Ku$ satisfies the algebraic Bianchi identity and \eqref{eq1}. Let $\Ri_K:V\too V$ be its Ricci operator. The main result here (see Proposition \ref{pr2} ) is that $\Ri_K$ has  at most two non-real complex eigenvalues and if we note  $\la_1,\ldots ,\la_r$ the non-zero real eigenvalues of $\Ri_K$, then one of the following situations is checked:
	\begin{enumerate}
	\item $\Ri_K$ only has real eigenvalues
and $V$  splits orthogonally:
	\begin{equation}\label{eq2} V=V_0\oplus V_{\la_1}\oplus\ldots\oplus V_{\la_r}, 
	\end{equation} where $V_{\la_i}=\ker(\Ri_K-\la_i\mathrm{Id}_V)$ and  $V_0=\ker(\Ri_K^2)$.
\item $\Ri_K$ has    two  non-real complex eigenvalues conjugated
  $z$  and $\bar{z}$. Then $V$ splits orthogonally:
	\begin{equation}\label{eq22} V=V_0\oplus V_c\oplus V_{\la_1}\oplus\ldots\oplus V_{\la_r}, 
	\end{equation}  
	where $V_{\la_i}=\ker(\Ri_K-\la_i\mathrm{Id}_V)$, $V_0=\ker(\Ri_K)$ and $V_c=\ker(\Ri_K^2-(z+\bar{z})\Ri_K+z\bar{z}\mathrm{Id}_V)$.
	
	Moreover, $\dim V_{\la_i}\geq2$, $\dim V_c=4$, $\Ku(V_{\la_i},V_{\la_j})=\Ku(V_0,V_{\la_i})=0$ for $i\not=j$,  $\Ku(u,v)(V_{\la_i})\subset V_{\la_i}$ and $\Ku(u,v)(V_{0})\subset V_{0}$.
	\end{enumerate}
	This reduces the study of semi-symmetric algebraic curvature tensors to the ones which are Einstein's ($\Ri_K=\la\mathrm{Id}_V$) or the ones which are Ricci's isotropic  ($\Ri_K\not=0$ and $\Ri_K^2=0$) or on the four dimensional neutral space such that $V=\ker(\Ri_K^2-(z+\bar{z})\Ri_K+z\bar{z}\mathrm{Id}_V)$.
	%%%%%%%%%%%%%%%%%%%%%%%%%%%%%%%%%%%%%%%%%%%%%%%%%%%%
	\item In second main proposition (see Proposition\ref{pr43}), we  give the list of the semi-symmetric curvature tensor on a four dimensional neutral space $(V,\prs)$.
	\item  In \cite{derd}, Derdzinsky gave a classification of four dimensional pseudo-Riemannian Einstein manifolds whose curvature treated as a complex linear operator is diagonalizable and has constant eigenvalues. In \cite{komrakov}, komrakov gave all   Lie subalgebra of $so(2,2)$. Based on these  results, we prove the following two results.
	
	\begin{theo}\label{theo2} A  four-dimensional Einstein neutral manifold with non null  scalar curvature is semi-symmetric if and only if it is locally symmetric.
		
	\end{theo}

	\begin{theo}\label{main}  Let $M$ be a    simply connected homogeneous semi-symmetric  4-dimensional  neutral manifold. If the Ricci tensor of $M$ has a non zero eigenvalue  in $\C$, then $M$ is symmetric and in this case it is a product of a space of constant curvature and a Cahen-Wallach space or it admits a complex structure.
	\end{theo} 
	
	We start in Section \ref{section4bis} by proving Theorem \ref{main} when $M$ is Lie group endowed with a left invariant  metric.   In Section \ref{section3}, we prove Theorems \ref{theo2} and \ref{main}.

	\item Having Theorem \ref{main} in mind, to complete the classification of  simply connected four-dimensional homogeneous semi-symmetric neutral manifolds, we determine
	all simply connected   four-dimensional semi-symmetric homogeneous   neutral no flat  manifolds with isotropic Ricci curvature  or flat Ricci.  
	\end{enumerate}
	
 The computations in Sections \ref{section4}, \label{section4bis} and \ref{section7} have been performed using a computation software.

     Note that a neutral space is a space equipped with metric of signature $(n,n)$. 

%%%%%%%%%%%%%%%%%%%%%%%%%%%%%%%%%%%%%%%%%%%%%%%%%%%%%%
%%%%%%%%%%%%%%%%%%%%%%%%%%%%%%%%%%%%%%%%%%%%%%%%%%%%%%
%%%%%%%%%%%%%%%%%%%%%%%%%%%%%%%%%%%%%%%%%%%%%%%%%%%%%%                    

\section{Semi-symmetric curvature tensors on pseudo-riemannian vector spaces} \label{section2}

In this section,  we prove the first result listed in the introduction.

Let $(V,\prs)$ be a $n$-dimensional pseudo-Riemannian  vector space. We identify $V$ and  its dual $V^*$ by the means of $\prs$. This implies that the Lie algebra $V\otimes V^*$ of endomorphisms of $V$ is identified with $V\otimes V$, the Lie algebra $\mathrm{so}(V,\prs)$ of skew-symmetric endomorphisms is identified with $V\wedge V$ and the space of symmetric endomorphisms is identified with $V\vee V$ (the symbol $\wedge$ is the outer product and $\vee$ is the symmetric product). For any $u,v\in V$, 
\[ (u\wedge v)w=\langle v,w\rangle u-\langle u,w\rangle v\esp (u\vee v)w=\frac12\left(\langle v,w\rangle u+\langle u,w\rangle v\right). \]
Through this paper, we denote by $A_{u,v}$ the endomorphism $u\wedge v$.
On the other hand, $V\wedge V$ carries  also a nondegenerate symmetric product also denoted by $\prs$ and given by
\[ \langle u\wedge v,w\wedge t\rangle:=\langle u\wedge v(w), t\rangle=
\langle v,w\rangle\langle u,t\rangle-\langle u,w\rangle\langle v,t\rangle. \]
We identify $V\wedge V$ with its dual by means of this metric.

We consider $ B $ the linear Bianchi application on the space $ P = (\we^2 V) \vee (\we^2 V) $ given by:

  \begin{equation}\label{eq1-4}B((a\we b)\vee (c\we d))=(a\we b)\vee (c\we d)+(b\we c)\vee (a\we d)+(c\we a)\vee (b\we d).\end{equation}
 %%%%%%%%%%%%%%%%%%%%%%%%%%%%%%%%%%%%%%%%%%%%%%
  Let be  $ \G $  a  subalgebra   of $so(V)$   and    the action of $\G$ on $P$ given  by $$ A.T : \quad (u\we v)\mapsto A.T(u\we v)=[A,T(u\we v)]-T(A(u)\we v)-T(u\we A(v)),$$  for all $(A,T)\in \G\times P $. 
  
 We  pute:
  \[ R(\G):=\ker(B_{/\G})=\{T \in \G\vee \G /~~ B(T)=0\} \]and
    \[\G_{sym}=\{T\in R(\G) / \G.T=0\}.\]
 -The space  $R(\G)$  is called space of {\it curvature tensor } of type $\G$ and     we say   {\it curvature tensor } of $V$ each element  of $R(so(V))$.\\
   -The space $\G_{sym}$ is called   space of  symmetric curvature tensor  of  type $\G$.

      According to the identifications mentioned above, we obtain:\\
A {\it curvature tensor} on    $(V,\prs)$ is a   $\mathrm{K}\in (V\wedge V)\vee (V\wedge V)$ (i.e. $\Ku$ is a symmetric endomorphism of $V\wedge V$) satisfying the algebraic Bianchi's identity:
 \[\mathrm{K}(u, v)w+\mathrm{K}(v, w)u+\mathrm{K}(w, u)v=0,\quad u,v,w\in V.\]
The Ricci curvature associated to $\mathrm{K}$ is the symmetric bilinear form on $V$ given by $\ric_K(u,v)=\tr(\tau(u,v))$, where $\tau(u,v):V\too V$ is given by $\tau(u,v)(a)=\mathrm{K}(u, a)v$. The  Ricci operator is the symmetric endomorphism $\mathrm{Ric}_{\Ku}:V\too V$ given by  $\langle
\mathrm{Ric}_K(u),v\rangle =\ric_{\Ku}(u,v)$.  We call ${\Ku}$ Einstein (resp. Ricci isotropic) if $\Ri_K=\la\mathrm{Id}_V$ (resp. $\Ri_{\Ku}\not=0$ and $\Ri_{\Ku}^2=0$).
 Note that if $\Ku=(u\wedge v)\vee(w\wedge t)$ then 
\[ \ric_K=\langle u,w\rangle t\vee v+\langle v,t\rangle u\vee w-\langle v,w\rangle t\vee u-
\langle u,t\rangle v\vee w.\]
We denote by  ${\mathfrak{h} }(\Ku)$ the vector subspace of $V\wedge V$ image of ${\Ku}$, i.e., ${\mathfrak{h} }(\Ku)=\mathrm{span}\{\Ku(u,v)/\;u,\;v\in V\}$. The Lie algebra genrated by  ${\mathfrak{h} }(\Ku)$ is called  {\it primitive holonomy algebra} of $\Ku$.\\
A curvature tensor $\mathrm{K}$ is called {\it semi-symmetric} if it is invariant by ${\mathfrak{h} }(\Ku)$, i.e.,
\[\mathrm{K}(u, v).\mathrm{K}=0,\quad \forall (u,v)\in V^2.\]
This is equivalent to
\begin{equation}\label{semi}
[\mathrm{K}(u, v),\mathrm{K}(a, b)]=\mathrm{K}(\mathrm{K}(u, v)a,b)+\mathrm{K}(a,\mathrm{K}(u,v)b),\quad \forall (u,v,a,b)\in V^4.
\end{equation} In this case, ${\mathfrak{h} }(\Ku)$ is a Lie subalgebra of $\mathrm{so}(V,\prs)$ and it is   {\it primitive holonomy algebra} of $\Ku$.
\begin{rem}
  If $\;\mathrm{K}\;$ is semi-symmetric, then its Ricci operator is also invariant by ${\mathfrak{h} }(\Ku)$, i.e., 
\begin{equation}\label{semi1} \mathrm{K}(u,v)\circ\mathrm{Ric}_\Ku=\mathrm{Ric}_\Ku\circ \mathrm{K}(u,v),\quad \forall (u,v)\in V^2.
\end{equation} The Ricci operator is said {\it semi-symmetric } if it is satisfaying \ref{semi1}.
\end{rem}

  In \cite{boubel}, C. Boubel  gave the following theorem:

\begin{theo}\label{boubel}(\cite{boubel})
  Let $(M,g)$ a pseudo-Riemannian  manifold with   parallel Ricci (i.e, $\na.\Ri=0$) and let  $\chi$ be minimal polynomial of $\Ri$. Then,   the following properties are checked:
   \begin{enumerate}
            \item $\chi=\Pi_iP_i$ with :\begin{enumerate}
            \item [$\bullet$] $\forall  i\neq j$, $P_i\we P_j=1 $ (i.e, $P_i$ and $P_j$ are mutually prime),
     \item [$\bullet$] $\forall  i$, $P_i$ is irreducible or  $P_i=X^2$.
   \end{enumerate}
   \item There is a canonical family $(M_i)_i$ of pseudo-Riemannian manifolds such that the minimal
polynomial of $\Ri_i = \Ri_{M_i}$ on each $M_i$ is $P_i$, and a local isometry $f$ mapping the
Riemarmian product $\Pi M_i$ onto M. $f$ is unique up to composition with a product of
isometries of each factor $M_i$. If $M$ is complete and simply connected, $f$ is an isometry.

That is to say, $M$ splits canonically into a Riemannian product, with factors $M_i$ of one
of the four following types- we denote by $P_i$ the minimal polynomial of $\Ri_i$, the Ricci
endomorphism of $M_i$:
\begin{enumerate}
\item if $P_i = (X - \al_i)^k$ with $\al_i  \neq0$, then $k = 1$, i.e. $M_i$ is Einstein,
\item if $P_i = X^k$, then $k \leq 2$, so
\begin{enumerate}
\item[i)] either $k = 1$, i.e. $M_i$ is Ricci-flat, [which is a particular case of Einstein],
\item[ii)] or Ricci is nilpotent of index 2,  \end{enumerate}
\item if $P_i = (X^2 + p_iX + q_i)^k$ (power of an irreducible), then $k = 1$, so $\Ri_i$ has no nilpotent
part but is not diagonalizable on $\R$. Then $M_i$ is  a complex
Riemannian manifold, which is Einstein for this structure.\\
The last two types do not appear in the Riemannian case.
        \end{enumerate}  \end{enumerate}
\end{theo}

In the proof of the first result of this theorem, \textbf{C. Boubel} used only the following algebraic hypothesis: On each tangent space $ T_xM $ at the point $ x \in M $, the Ricci operator $ \Ri_x $ commutes with   each endomorphisms $ R_x (u, v) $ for all $ u $, $ v \in T_xM $.  that is said Ricci operator is semi-symmetric which was verified  for  spaces with semi-symmetric curvature. Then, we get the following results:

 %%%%%%%%%%%%%%%%%%%%%%%%%%%%%%%%%%%%%%%%%%%%%%%%%%%%%%%%%%%%%%%%%%%%%%%%%%
   %%%%%%%%%%%%%%%%%%%%%%%%%%%%%%%%%%%%%%%%%%%%%%%%%%%%%%%%%%%%%%%%%%%%%%
 \begin{pr}\label{pr1-3-1}
Let $\Ku$ be a  semi-symmetric curvature tensor on the   pseudo-Riemannian  space  $(V,\prs)$  and let   $\chi$ be a minimal polynomial  of    $\Ri_{\Ku}$. Then, the following properties are checked:
   \begin{enumerate}
            \item $\chi=\Pi_iP_i$,  with;\begin{enumerate}
            \item [$\bullet$] $\forall  i\neq j$, $P_i\we P_j=1 $ (i.e, $P_i$ and $P_j$ are mutually prime),
     \item [$\bullet$] $\forall  i$, $P_i$ is irreducible or  $P_i=X^2$.
   \end{enumerate}
\item $V$ splits orthogonally: \begin{equation}\label{dec-V-Ric} V=V_0\oplus V_1\oplus\ldots\oplus V_r,\end{equation}   where $V_0=\ker((\mathrm{Ric}^2))$ and    $V_i=\ker(P_i(\mathrm{Ric}))$. \\
Moreover,    the following situations is verified:
\begin{enumerate}
\item[a)] for all  $u,v\in V$  and $i\in\{0,\ldots,r\}$, $V_i$ is ${\mathfrak{h} }(\Ku)$-invariant,,
\item[b)]  for all  $i,j\in\{0,\ldots,r\}$ with  $i\not=j$, $\Ku_{|V_i\wedge V_j}=0$,
\item[c)]  for all  $i=1,\ldots,r$, $\dim V_i\geq2$.
\end{enumerate}
\end{enumerate}
 \end{pr}

 We give now the main result of this section which gives a useful decomposition of semi-symmetric curvature tensors in a pseudo-Riemannian  space  with metric of signature $(2,n)$. 
 
\begin{pr}\label{pr2}   Let $\Ku$ be a  semi-symmetric curvature tensor on the   pseudo-Riemannian  space  $(V,\prs)$   with  metric of signature $(2,n)$ such that $n\geq2$.  Then the Ricci curvature $\Ri_{\Ku}$ admits at most two non-real eigenvalues. Denote by $\al_1,\ldots,\al_r$ the non-zero real eigenvalues and $V_1,\ldots,V_r$ the corresponding eigenspaces. Then one of the following situations is verified:\begin{enumerate}
    \item   $\Ri_{\Ku}$   has two non-real eigenvalues $z$ and $\bar{z}$ and $V$ splits   orthogonally  \[V= V_c\oplus V_0\oplus \ldots\oplus V_r,\] where  $V_0=\ker(\mathrm{Ric})$ and  $V_c=\ker(\Ri_{\Ku}^2-(z+\bar{z})\Ri_{\Ku}+|z|^2I) $.\\
         Moreover,    $~~\dim(V_c)=4$ and $V_i$ is a Riemannian  semi-symmetric space for all  $i\geq 0$.\\In this case, $\Ri_{\Ku}$   is said complex Ricci.
    \item  $\mathrm{Ric}_{\Ku}$ has only  real eigenvalues and   $V$ splits   orthogonally: $$V= V_0\oplus V_1\oplus \ldots\oplus V_r,~~ \text{where}~~~~V_0=\ker(\Ri_{\Ku}^2).$$

\end{enumerate}

\end{pr}
\begin{proof} $   $

We will only show the following result:
If $\Ri=\mathrm{Ric}_{\Ku}$  admits a non real eigenvalue  $z$, then,   $z$ and $\bar{z}$ are the only non-real eigenvalues of $\Ri$ and  $V_c=\ker(\Ri^2-(z+\bar{z})\Ri+|z|^2I) $ is the neutral  space of dimension $4$. The other  results are easy to proof.

%%%%%%%%%%%%%%%%%%%%%%%%%%%%%%%%%%%%%%%%%%%%%%%%%%%%%%%%%%%%%%%%%%%%%%%%%%%%%%%%%%%%%%%%%%%%%%%%%%%%%%%%%%%%%%
 Now, we  suppose that  $\Ri$  admits a non-real eigenvalue  $z$ with the associated caractestic   subspace    $V_c=\ker(\Ri^2-(z+\bar{z})\Ri+|z|^2I) $, that subspace is a  pseudo-Riemannian  semi-symmetric of even dimension and
  necessarily, $V_c$    admits the metric of signature $(2,2p)$, 
   otherwise, we have  $V_c$ is a non degenerete space then, one and only one of  two following situations is checked:
    \begin{enumerate}
      \item $V_c$ is a  Riemannian space . So    $\Ri$ is a symmetric endomorphism and it's diagonalizable  admitting only the real eigenvalues,
      \item   $V_c$ is a  Lorentzian  semi-symmetric space. According to \cite{benromane},   $\Ri$ has only the real eigenvalues.
    \end{enumerate}
    Which is impossible in both situations.

 So,
       $(V_c)^{\perp}$  will be a Riemannian space which  $\Ri$  has only real eigenvalues.

        Now, we  choose a non-zero isotropic vector $e$ in $V_c$.
         Then, $(e,\Ri(e))$ is a free family in $V_c$, otherwise, $e$  will be an  eigenvector  associated  of a real eigenvalue of  $ \Ri$  on  $V_c$, which is impossible.

        On the other hand, the subspace $ V_c^e = span\{e, \Ri (e) \} $ generated by $ e $ and its image $ \Ri (e) $,  is totally isotropic if and only if $ e $ and $ \ Ri (e) $ are orthogonal. Then, one of the two following situations is verified:

         \begin{enumerate} \item[a)] If $e$ and $\Ri(e)$ are not  orthogonl (i.e $\langle e,\Ri(e)\rangle \not=0)$, then,  $V_c^e$ will be   a  Lorentzian space   stable by $\Ri$ and let  $\overline{V_c^e}$ be a subspace such that  $V_c$ splits orthogonally    $V_c=V_c^e\oplus \overline{V_c^e}$. Then $\overline{V_c^e}$ is a Lorentzian  subspace invariant by  $\Ri$ then it's semi-symmetric and necessarily, $\overline{V_c^e}$  is  of dimension $2$ and  $\dim({V_c})=4$.

          \item[b)] Now, if    $V_c^e=span\{e,\Ri(e)\}$ is totally  isotropic.\\
           We choose  $\bar{e}$ a dual vector of $\Ri(e)$  in  $V_c$. So, $\Ri(\bar{e})$ is a dual vector of $e$ and one of the two following situations is verified:\begin{enumerate}
                                                                   \item[b1)] $\bar{e}$ and $\Ri(\bar{e})$ are duals vectors. Then, $\overline{V_c^e}=vect\{\bar{e},~\Ri(\bar{e})\}$ is a Lorentzian subspace of $V_c$, stable by $\Ri$  and we come back to the case (a).
                                                                   \item[b2)] $\bar{e}$ and $\Ri(\bar{e})$ are orthogonals. Then, $\overline{V_c^e}=vect\{\bar{e},~\Ri(\bar{e})\}$ is a totally isotropic subspace of  $V_c$ and a dual subspace of  $V_c^e$ and
        $V_c= V_c^e \oplus\overline{V_c^e}$. This completes the proof of the proposition.

                                                                 \end{enumerate} \end{enumerate}

         \qedhere

\end{proof}

%%%%%%%%%%%%%%%%%%%%%%%%%%%%%%%%%%%%%%%%%%%%%%%%%%%%%%%%%%%%%%%%%%%%%%%%%%%%%%%%%%%%%%%%%%%%
%%%%%%%%%%%%%%%%%%%%%%%%%%%%%%%%%%%%%%%%%%%%%%%%%%%%%%%
%%%%%%%%%%%%%%%%%%%%%%%%%%%%%%%%%%%%%%%%%%%%%%%%%%%%%%%%%%%%%%%%%%%%%%%%%%%%%%%%%%%%%%%%%%%%%%%%%

This proposition reduces  the determination of semi-symmetric curvature tensors on  vector space  equiped  with metric of signature $(2,n)$ to the determination of three classes of semi-symmetric curvature tensors: Einstein semi-symmetric curvature tensors,   semi-symmetric curvature with Ricci isotropic, semi-symmetric curvature tensor  with Ricci admiting two non-real eigenvalues conjugated on four dimensional neutral space.
%----------------------------------------------------------------------------
%------------------------------------------------------
%%%%%%%%%%%%%%%%%%%%%%%%%%%%%%%%%%%%%%%%%%%%%%%%%%%%%%%%%%%%%%
	%%%%%%%%%%%%%%%%%%%%%%%%%%%%%%%%%%%%%%%%%%%%%
	%%%%%%%%%%%%%%%%%%%%%%%%%%%%%%%%%%%%%%%%%%%%%
   \section{  Semi-symmetric curvature tensor on four dimensional neutral space }\label{section4}
   
In this section we give  the classification of semi-symmetric curvature tensors on four dimensional  neutral vector spaces. 

 The idea  is the following: Let $\Ku$ be  curvature tensor on   pseudo-Riemannian vector space $(V,\prs)$. $\Ku$ is semi-symmetric curvature tensors if only if $\mathfrak{h}(\Ku).\Ku=0$ and the space $\mathfrak{h}(K)$ is  a Lie subalgebra of $\mathfrak{so}(V,\prs)$. So, one way to determine all semi-symmetric curvature tensors is given all proper subalgebras of $\mathfrak{so}(V,\prs)$ and for each one of them, say $\G$, determine all the curvature tensors $\Ku\in \G_{sym}$,i.e. $\Ku\in\G\vee \G$ satisfaying $B(\Ku)=0$ and  $\G.\Ku=0$ where $B$ is the linear Bianchi application.\\ 
%%%%%%%%%%%%%%%%%%%%%%%%%%%%%%%%%%%%%%%%%%%%%%%%%%%%%%%%%%%%%%%%%%
First step, according of the classification of Lie subalgebra of  $\mathfrak{so}(2,2)$  given  by Komrakov  in\cite{komrakov} and for each   subalgebra  $\G$ of  $\mathfrak{so}(2,2)$,  we   give  
 all curvature  tensor and  all symmetric  curvature of type   $\G$.
\\ In second step,   we give a   list  of  non flat semi-symmetric  curvature  tensors on the  four dimensional neutral  vector space $(V,\prs)$ by giving the list of Lie subalgebras $\G\neq \{0\}$ of  $\mathfrak{s}o(V)$ such that  $\G=\G_{sym}$.
%%%%%%%%%%%%%%%%%%%%%%%%%%%%%%%%%%%%%%%%%%%%%%%%%%%%%%%%%%%%%%%%%%%%%%%%%%%%%%%%%
   \begin{theo} For each  Lie subalgebra  $\G$ of  $\mathfrak{so}(2,2)$,  the space $R(\G)$ of all curvature  tensor of type $\G$ and the space $\G_{sym}$ of all symmetric  curvature  tensor of type $\G$ are the following:

 \begin{enumerate}\item{$\dim \G=1$}:
 \begin{enumerate}
 \item $\G:1.1^1=\R \{A_{x,z}+a.A_{y,t}\}$, with $\langle x,z\rangle=\langle y,t\rangle=1$, $a\in [0,1]$\\
       If  $a=0$,  we get;  $R(\G)=\G_{sym}=\R\{ A_{x,z}\vee A_{x,z}\}.$  \\
        Otherwise, $ R(\G)=0$,
 \item  $\G:1.1^2=\R \{A_{x,z}+a.A_{y,t}\}$,  with  $\langle x,x\rangle=-\langle y,y\rangle=\langle z,z\rangle=-\langle t,t\rangle=1$, $a\in [0,1]$\\
     If  $a=0$,   we get;  $ R(\G)=\G_{sym}=\R\{A_{x,z}\vee A_{x,z}\}$.\\     Otherwise,    $ R(\G)=0$,
 \item $\G:1.2^1=\R \{A_{x,z}+A_{x,t}+A_{y,t}\}$,  with  $\langle x,z\rangle=\langle y,t\rangle=1$, \\
        $\; R(\G)=0$.
        \item $\G:1.2^2=\R \{A_{x,y}+A_{x,t}+A_{y,z}\}$,  with  $\langle x,z\rangle=\langle y,t\rangle=1$, \\
        $ \;R(\G)=0$.
 \item  $\G:1.3^1=\R A_{x,y}$,  with  $\langle x,t\rangle=\langle y,z\rangle=1$, \\
      $ R(\G)=\G_{sym}=\R.A_{x,y}\vee A_{x,y}$.
 \item  $\G:1.4^1=\R A_{x,y}$,  with  $\langle x,z\rangle=-\langle y,y\rangle=\langle t,t\rangle=1$, \\
       $ R(\G)=\G_{sym}=\R A_{x,y}\vee  A_{x,y}$.
  \item   $\G:1.1^5=\R\{ \cos(\phi)(A_{x,t}+A_{y,z})+\sin(\phi)(A_{y,z}+A_{t,y})\}$  with  $\langle x,t\rangle=\langle y,z\rangle=1$  and     $\phi\in ]0,\frac{\pi}{4}]$, \\
         $ R(\G)=0$,
       \item   $\G:1.1^6=\R\{ \cos(\phi)(A_{x,t}+A_{z,y})+\sin(\phi)(A_{x,z}+A_{y,t})\}$  with  $\langle x,z\rangle=\langle y,t\rangle=1$  and     $\phi\in ]0,\frac{\pi}{4}[$, \\
       $ \;  R(\G)=0$,
  \end{enumerate}
  \item{$\dim\G=2$}:
    \begin{enumerate}
 \item  $\G:2.1^1=vect\{ A_{x,z},~~A_{y,t}\}$,  with  $\langle x,z\rangle=\langle y,t\rangle=1$,\\
       $ R(\G)=\G_{sym} =span\{ A_{x,z}\vee  A_{x,z},~A_{y,t}\vee  A_{y,t} \} $.
 \item $\G:2.1^3=span\{ A_{x,z},~~A_{y,t}\}$,  with  $\langle x,x\rangle=-\langle y,y\rangle=\langle z,z\rangle=-\langle t,t\rangle=1$,\\
       $ R(\G)=\G_{sym} =span\{ A_{x,z}\vee  A_{x,z},~A_{y,t}\vee  A_{y,t}\} $.
 \item  $\G:2.1^4=span\{ \pi_1=A_{x,z}+A_{t,y},~~\pi_2=A_{x,t}+A_{y,z}\}$,  with  $\langle x,t\rangle=\langle y,z\rangle=1$,\\
       $ R(\G)=\G_{sym}=span\{  (\pi_1\vee \pi_1-\pi_2\vee \pi_2),~\pi_1\vee \pi_2  \}$.\\
\item $\G:2.2^1=span\{ \pi_1=A_{x,z}+aA_{y,t},~~\pi_2=A_{x,t}\}$  with  $\langle x,z\rangle=\langle y,t\rangle=1$  and     $a\in [-1,1]$,\\
        If  $a=0$, we get   $ R(\G)=span\{ \pi_1\vee \pi_1,~\pi_2\vee \pi_2,~\pi_1\vee \pi_2\}$, and     $\G_{sym}=0$.\\
          If  $a=1$,  we get   $ R(\G)=span\{ \pi_2\vee \pi_2,~~\pi_1\vee \pi_2 \}=\G_{sym}$.\\
           If  $a\neq1$  and     $a\neq 0$,  we get   $ R(\G)=span\{  \pi_2\vee \pi_2,~~\pi_1\vee \pi_2 \}$ and      $\G_{sym}=0$.
\item   $\G:2.2^3=span\{ \pi_1= \cos(\phi)(A_{x,t}+A_{z,y})+\sin(\phi)(A_{x,z}+A_{y,t}),~~\pi_2=A_{x,y}\}$ with  $\langle x,z\rangle=\langle y,t\rangle=1$  and     $\phi\in ]0,\frac{\pi}{2}[$, \\
      $ R(\G)=span\{  A_{x,y}\vee A_{x,y},~~A_{x,y}\vee \pi_1 \}$    and     $\G_{sym}=0$.
    \item $\G:2.3^1=span\{ \pi_1=A_{x,z}+A_{x,y}+A_{t,y},~~\pi_2=A_{x,t}\}$,  with  $\langle x,z\rangle=\langle y,t\rangle=1$,\\
       $ R(\G)=\R.\pi_2\vee \pi_2$  and     $\G_{sym}=0$.\\
\item $\G:2.4^1=span\{ A_{x,z}, A_{x,y}\}$,  with  $\langle x,z\rangle=-\langle y,y\rangle=\langle t,t\rangle=1$ ,\\
       $ R(\G)=span\{ A_{x,z}\vee A_{x,z},~~A_{y,x}\vee A_{y,x},~A_{x,z}\vee A_{y,x}\}$    and     $\G_{sym}=0$.
 \item $\G:2.5^1=span\{ A_{x,t}, A_{x,y}\}$,  with  $\langle x,z\rangle=\langle y,t\rangle=1$ ,\\
       $ R(\G)=\G_{sym}=span\{  A_{x,y}\vee A_{x,y},~~A_{x,t}\vee A_{x,t},~~A_{x,y}\vee A_{x,t} \}  .$
  \end{enumerate}
  \item{$\dim\G=3$}:
  \begin{enumerate}
    \item   $\G:3.1^1=span\{ A_{x,z}, A_{x,t}, A_{y,t}\}$,  with  $\langle x,z\rangle=\langle y,t\rangle=1$ ,\\
       $ R(\G)=span\{A_{x,z}\vee A_{x,z},~~A_{x,t}\vee  A_{x,t},~~A_{y,t}\vee  A_{y,t},~~A_{x,t}\vee  A_{y,t},~~A_{x,z}, A_{x,t}\}$,\\ and      $\G_{sym}=0$

  \item $\G:3.2^1=span\{ \pi_1=A_{x,z}+aA_{y,t},~~\pi_2=A_{x,t},~~\pi_3=A_{x,y}\}$, with  $\langle x,z\rangle=\langle y,t\rangle=1$  and     $a\geq0$,\\
        If  $a=0$,   we get   $R(\G)=span\{\pi_1\vee \pi_1,~~\pi_2\vee \pi_2,~~\pi_3\vee \pi_3,~~\pi_2\vee \pi_3,~~\pi_1\vee \pi_2
       ,~~\pi_1\vee \pi_3 \} $ and     $\G_{sym}=0$.\\
             Otherwise,  $ R(\G)=span\{   \pi_2\vee \pi_2,~~\pi_3\vee \pi_3,~~\pi_2\vee \pi_3,~\pi_1\vee \pi_2,~\pi_1\vee \pi_3\},$\\
   and         if more  $a\neq1$,   we get     $\G_{sym}=\R\{ \pi_2\vee \pi_2\}$,     otherwise,  we get;  $\G_{sym}=0$ ,\\

  \item $\G:3.3^1=span\{ \pi_1=A_{y,t},~~\pi_2=A_{x,t},~~\pi_3=A_{x,y}\}$  with  $\langle x,z\rangle=\langle y,t\rangle=1$,\\
       $ R(\G)=span\{ \pi_1\vee \pi_1,~~\pi_2\vee \pi_2,~~\pi_3\vee \pi_3,~~\pi_2\vee \pi_3,~~\pi_1\vee \pi_2,~~\pi_1\vee \pi_3\}$,  
       $\G_{sym}=\R.\pi_2\vee \pi_3$.
    \item $\G:3.4^1=span\{ \pi_1=A_{x,z}+A_{t,y},~~\pi_2=A_{x,t},~~\pi_3=A_{y,z}\}$  with  $\langle x,z\rangle=\langle y,t\rangle=1$,\\
        $ R(\G)=span\{ (\pi_1\vee \pi_1-2\pi_2\vee \pi_3),~~\pi_2\vee \pi_2,~~\pi_3\vee \pi_3,~~\pi_1\vee \pi_2,~~\pi_1\vee \pi_3 \}$,    $\G_{sym}=0$.

        \item $\G:3.5^1=span\{ \pi_1=A_{x,z},~~\pi_2=A_{x,y},~~\pi_3=A_{y,z}\}$, with  $\langle x,z\rangle=\langle y,y\rangle=-\langle t,t\rangle=1$,
   $ R(\G)=span\{ (\pi_1\vee \pi_1),~~(\pi_2\vee \pi_2),~~(\pi_3\vee \pi_3),~~(\pi_1\vee \pi_2),~~\pi_1\vee \pi_3,~~\pi_2\vee \pi_3 \}$, \\  and     $\G_{sym}=\R.(\pi_1\vee \pi_1+2\pi_2\vee \pi_3)$.
    \end{enumerate}
    \item{$\dim\G=4$}:
     \begin{enumerate}
    \item   $\G:4.1^1=span\{ A_{x,y}, A_{x,z}, A_{x,t}, A_{y,t}\}$  with  $\langle x,z\rangle=\langle y,t\rangle=1$ ,\\
       $ R(\G)=span\{A_{x,y}\vee A_{x,y},~~A_{x,z}\vee  A_{x,z},~~A_{x,y}\vee  A_{x,z},~~  A_{x,t}\vee  A_{x,z},~~ A_{x,t}\vee  A_{x,t},~~A_{y,t}\vee  A_{x,t},~~A_{x,y}\vee  A_{x,t},~~ A_{y,t}\vee  A_{y,t},~~A_{x,y}\vee  A_{y,t} \}$,\\
        $\G_{sym}=0$

         \item   $\G:4.2^1=span\{ A_{x,t}, A_{x,z}, A_{y,t}, A_{y,z}\}$  with  $\langle x,z\rangle=\langle y,t\rangle=1$ ,\\
       $ R(\G)=span\{ A_{x,z}\vee A_{x,z},~A_{x,t}\vee  A_{x,t},~A_{y,z}\vee  A_{y,z},~A_{y,t}\vee  A_{y,t},~A_{x,z}\vee  A_{x,t},~A_{x,z}\vee  A_{y,z},~(A_{x,z}\vee  A_{y,t}+A_{x,t}\vee  A_{y,z}),~A_{x,t}\vee  A_{y,t},~A_{y,z}\vee  A_{y,t} \}$,\\  and    $\G_{sym}=\R\{A_{x,z}\vee A_{x,z}+A_{y,t}\vee  A_{y,t}+A_{x,z}\vee  A_{y,t}+A_{x,t}\vee  A_{y,z}\}$

         \item   $\G:4.3^1=span\{\pi_1= A_{x,y}, \pi_2=A_{x,t}, \pi_3= A_{t,y}+ A_{x,z}, \pi_4=A_{y,z}\}$  with  $\langle x,z\rangle=\langle y,t\rangle=1$ ,\\
       $ R(\G)=span\{ \pi_1\vee \pi_1,~\pi_2\vee \pi_2,~(\pi_3\vee \pi_3-2\pi_2\vee \pi_4),~\pi_4\vee \pi_4
       ,~\pi_1\vee \pi_2,~~\pi_1\vee \pi_3 ,~\pi_1\vee \pi_4,~\pi_2\vee \pi_3,~\pi_3\vee \pi_4 \}$,\\ and     $\G_{sym}=\R\{\pi_1\vee \pi_1 \}$.
       \end{enumerate}
       \item-{$\dim\G=5$ or $6$}:
     \begin{enumerate}
            \item   $\G:5.1^1=span\{ A_{x,y},A_{x,z}, A_{x,t},  A_{y,z}, A_{y,t}\}$  with  $\langle x,z\rangle=\langle y,t\rangle=1$ ,\\$\dim( R(\G))=14$   and     $\G_{sym}=0$.
     \item   $\G:6.1^1=so(2,2)$, we get   $\dim( R(\G))=19$ \\$\G_{sym}=\R(A_{x,z}\vee A_{x,z}+A_{y,t}\vee  A_{y,t}+A_{x,z}\vee  A_{y,t}+2.A_{x,t}\vee  A_{y,z}+2.A_{x,y}\vee
          A_{t,z})$  with   $\langle x,z\rangle=\langle y,t\rangle=1$
       \end{enumerate}
   \end{enumerate}
\end{theo}

 \begin{pr}{\label{pr43}}  Let $\Ku$ be a semi-symmetric curvature tensor  on  the  four dimensional neutral  vector space $(V,\prs)$.   Then,   there is the basis   $(x,y,z,t)$ of  $V$ such that the one of following  situations is checked:
\begin{enumerate}
\item  $\dim \h( \Ku)=1$ and $\Ku$  has one of the following forms:

 \begin{enumerate}
 \item $ \Ku=b. A_{x,z}\vee A_{x,z}$  and        $\;\Ri=-2b.z\vee x$, where $b\in \R^*$\\  with  $\langle x,z\rangle=\langle y,t\rangle=1$.
 \item  $ \Ku=b. A_{x,z}\vee A_{x,z}$   and    $\Ri=b.(x\vee x+z\vee z)$  where $b\in \R^*$,\\  with  $\langle x,x\rangle=-\langle y,y\rangle=\langle z,z\rangle=-\langle t,t\rangle=1$.
 \item   $ \Ku=a.A_{x,y}\vee A_{x,y}$,  $\Ri=0$ and    $a\in\R^*$,\\  with  $\langle x,t\rangle=\langle y,z\rangle=1$.
 \item  $ \Ku=a A_{x,y}\vee A_{x,y}$   and     $\Ri=-a. x \vee x$,\\ with $a\in \R^*$ and    $\langle x,z\rangle=-\langle y,y\rangle=\langle t,t\rangle=1$.
    \end{enumerate}
\item   $\dim \h( \Ku)=2$ and $\Ku$  has one of the following forms:

    \begin{enumerate}
 \item  $ \Ku= a.A_{x,z}\vee A_{x,z}+bA_{y,t}\vee  A_{y,t}$, $\Ri=-2(a.x\vee z+b.y {\vee }t)$,\\ with  $(a,b)\in\R^*\times \R^*$ and     $\langle x,z\rangle=\langle y,t\rangle=1$.
 \item  $ \Ku= a.A_{x,z}{\vee }A_{x,z}+bA_{y,t}{\vee } A_{y,t}$, $~~\Ri=a(x\vee x+z\vee z)-b(y \vee y+t\vee t)$,\\ with  $(a,b)\in\R^*\times\R^* $   and    $\langle x,x\rangle=-\langle y,y\rangle=\langle z,z\rangle=-\langle t,t\rangle=1$,\\

 \item    $ \Ku= a(\pi_1\vee \pi_1-\pi_2\vee \pi_2)+b.\pi_1\vee \pi_2,  ~~\Ri=-4a(x\vee t+y\vee z)+2b(y\vee t-x\vee z) $,\\ with $(a,b)\in\R^*\times\R^*$,
             $\pi_1=A_{x,z}+A_{t,y},~~\pi_2=A_{x,t}+A_{y,z}$  and    $\langle x,t\rangle=\langle y,z\rangle=1$,\\

\item $ \Ku= c\pi_2\vee \pi_2+d\pi_1\vee \pi_2, ~~\Ri=-2d.x\vee t $,\\  with $(c,d)\in\R^*\times\R^*$  and     $\pi_1=A_{x,z}+A_{y,t},~~\pi_2=A_{x,t}$  and    $\langle x,z\rangle=\langle y,t\rangle=1$.

 \item $ \Ku=a.A_{x,y}\vee A_{x,y}+b.A_{x,t}\vee A_{x,t}+c.A_{x,y}\vee A_{x,t}$, $\Ri=c.x\vee x$,\\ with  $(a,b,c)\in \R^*\times\R^*\times\R^*$ and      $\langle x,z\rangle=\langle y,t\rangle=1$.
  \end{enumerate}
     \item    $\dim\h( \Ku)=3$ and  there is $a\in\R^*$   such that   :\\ $ \Ku=a(A_{x,z}\vee A_{x,z}+2A_{x,y}\vee A_{y,z})$  and    $\Ri=-2a(2x\vee z+y\vee y)$,\\
         with   $\langle x,z\rangle=\langle y,y\rangle=-\langle t,t\rangle=1$ ,\\

    \item   $\dim\h( \Ku)=4$ and   there is $a\in\R^*$   such that:\\ $ \Ku=a(A_{x,z}\vee A_{x,z}+A_{y,t}{\vee } A_{y,t}+A_{x,z}{\vee }A_{y,t}+A_{x,t}{\vee }A_{y,z})$,  $\Ri=-3a(.x\vee z+y\vee t)$,\\ with  $\langle x,z\rangle=\langle y,t\rangle=1$.
 \item  {$\dim\h( \Ku)=6$} then, 
        $\h( \Ku)=so(2,2)$  and  there is $a\in\R^*$   such that:    $$ \Ku=a(A_{x,z}\vee A_{x,z}+A_{y,t}{\vee }A_{y,t}+A_{x,z}{\vee } A_{y,t}+2.A_{x,t}{\vee } A_{y,z}+2.A_{x,y}{\vee }A_{t,z}),\;\; \text{ and }\;\;\Ri=-6a(x\vee z+y\vee t),$$ with   $\langle x,z\rangle=\langle y,t\rangle=1$
       \end{enumerate}
 \end{pr}
In the following corollary, we give all curvature tensors on the  four dimensional neutral  vector space in the some particuler cases: Einstein, isotropic Ricci or complex Ricci.  

\begin{co}\label{co4-21}   Let $\Ku$ be a semi-symmetric curvature tensor  on  the  four dimensional neutral  vector space $(V,\prs)$.   Then,   there is the basis   $(x,y,z,t)$ of  $V$ such that:
\begin{enumerate}
\item   If  $ \Ku$ is the  Einstein curvature tensor with non zero scalar  curvature.   Then,  one of the  following  situations is checked::
 \begin{enumerate}
 \item  $\dim\h( \Ku)=6$, then;\\  $ \Ku=a(A_{x,z}\vee A_{x,z}+A_{y,t}{\vee }A_{y,t}+A_{x,z}{\vee } A_{y,t}+2.A_{x,t}{\vee } A_{y,z}+2.A_{x,y}{\vee }A_{t,z})$,\\ $\Ri=-6a(.x\vee z+y\vee t)$,\\ with $a\in\R^*$ and     $\langle x,z\rangle=\langle y,t\rangle=1$
 \item  $\dim\h( \Ku)=4$,  then,   \\  $ \Ku=a(A_{x,z}\vee A_{x,z}+A_{y,t}{\vee } A_{y,t}+A_{x,z}{\vee }A_{y,t}+A_{x,t}{\vee }A_{y,z})$ ,\\ with   $a\in\R^* $ and    $\langle x,z\rangle=\langle y,t\rangle=1$.
 \item  $\dim\h( \Ku)=2$,  then;   \\ $ \Ku= a.(A_{x,z}{\vee }A_{x,z}+A_{y,t}{\vee } A_{y,t})$ or $\Ku= a.(A_{x,y}{\vee }A_{x,y}+A_{z,t}{\vee } A_{z,t})$,\\ with $a\in\R^* $  and     $\langle x,x\rangle=-\langle y,y\rangle=\langle z,z\rangle=-\langle t,t\rangle=1$,
     \end{enumerate}
 \item     If  $ \Ku$ is a  Ricci  flat,  then,  there is  $(a,b) \in\R^* \times \R^* $  such that   ;\\  $ \Ku=a.A_{x,y}\vee A_{x,y}+b.A_{x,t}\vee A_{x,t}$,\\  with  $\langle x,z\rangle=\langle y,t\rangle=1$.

 \item  If  $ \Ku$ is an isotropic Ricci,  Then     the one of following  situations is checked:   
 \begin{enumerate}
 \item  $ \Ku=a A_{x,y}\vee A_{x,y}$   and     $\Ri=-a. x \vee x$,\\with $a\in\R^*$ and       $\langle x,z\rangle=-\langle y,y\rangle=\langle t,t\rangle=1$.
 \item $ \Ku= c\pi_2\vee \pi_2+d(\pi_1\vee \pi_2), ~~\Ri=-2d.x\vee t $,\\with  $(c,d)\in\R\times \R^*$ ,  $\pi_1=A_{x,z}+A_{y,t},~~\pi_2=A_{x,t}$  and    $\langle x,z\rangle=\langle y,t\rangle=1$.
 \item $ \Ku=a.A_{x,y}\vee A_{x,y}+b.A_{x,t}\vee A_{x,t}+c.A_{x,y}\vee A_{x,t}$, $\Ri=c.x\vee x$,\\ with $(a,b,c)\in\R^2\times \R^*$ and     $\langle x,z\rangle=\langle y,t\rangle=1$
  \end{enumerate}
\item   If  Ricci has a non-real  eigenvalue,   then, there is  $(a,b)\in\R\times \R^*$   such that   
    \[ \Ku= a(\pi_1\vee \pi_1-\pi_2\vee \pi_2)+b.\pi_1\vee \pi_2 ~~\text{and}~~
         \Ri=-4a(x\vee t+y\vee z)+2b(y\vee t-x\vee z) \] with   $\pi_1=A_{x,z}+A_{t,y},~~\pi_2=A_{x,t}+A_{y,z}$  and    $\langle x,t\rangle=\langle y,z\rangle=1$.
  \end{enumerate}

 \end{co}
%%%%%%%%%%%%%%%%%%%%%%%%%%%%%%%%%%%%%%%%%%%%%%
%%%%%%%%%%%%%%%%%%%%%%%%%%%%%%%%%%%%%%%%%%%%%%
%%%%%%%%%%%%%%%%%%%%%%%%%%%%%%%%%%%%%%%%%%%%%%
 \section{Semi-symmetric manifolds}

Let $(M,g)$ be  pseudo-Riemannian manifolds of dimension $n$ with Levi-Civita connexion $\na$  and Riemannian curvature $R$,   $ \mathfrak{ric}$ and $\Ri$ are the Ricci tensor and the Ricci     operator respectively and let ${\cal X}(M)$ be the space of all vetor fields on $M$.
 \begin{Def}
  $(M,g)$  has said   {\it  semi-symmetric} if, $R.R=0$, i.e,  $R$ verifies \begin{equation}\label{eq1-9va}
 [R(X,Y),R(Z,T)]=R(R(X,Y)Z,T)+R(Z,R(X,Y)T),\quad X,Y,Z,T\in {\cal X}(M).\end{equation}
   \end{Def}

Let $(M,g)$ be  pseudo-Riemannian  semi-symmetric manifolds. Then, for each point
 $m\in M$, the restriction $R_m$ of $R$ on tangent space $T_mM$ is  semi-symmetric curvature tensor. So, the minimal  polynomial of $\Ri_m$ is of the form $\X=\prod_i P_i$ such that  the polynomials $(P_i)_i$ are   mutually prim between them and for all $i$, $P_i$ is irreducible  or  $P_i=X^2$.   We can  suppose that  $P_i$ is irreducible for all  $i\geq1$ and  define  the distributions:  \[V_0(m):=\ker(\Ri_m^2)~~\text{ et}  ~~V_i(m):=\ker(P_i(\Ri_m))~~\text{ for all}~~i\geq 1,\] and we get the following proposition:

     \begin{pr}\label{pr121} The distributions $(V_i)_i$ have the following  proprietes:

     For all $i\neq j$, we get:
     \begin{equation}\label{l}
             \na_{V_j}V_i\subset V_i, \;\;  \na_{V_i}V_i\subset V_0+V_i, \;\;
             \na_{V_0}V_i\subset V_i, \;\;
              \na_{V_0}V_0\subset V_0, \;\;  \na_{V_i}V_0\subset V_0+V_i.
         \end{equation}

     \end{pr}

    \begin{proof}$\;$

Let $m\in M$. According the propositon.\ref{pr1-3-1},  we get: \begin{equation}\label{dec-Tm-ri} V_m:= T_mM=V_0(m)\oplus V_1(m)\oplus...\oplus V_r(m). \end{equation}

In first,  we show that for  $i\geq1$  and $X\in V_i^{\perp}$,   $\na_XV_i\subset V_i$:\\
 We choose $i\geq1$ and   $X\in V_i^{\perp}$. Then,  we get  $\; \na_X(R(V_i,V_i)V_i)\subset V_i$.

  Indeed:

  Let   $Y,\; Z,\; T\in V_i$. According the second identity of Bianchi, we get:

				\begin{eqnarray*} \na_X R(Y,Z,T)&:=&(\na_X R)(Y,Z)T\\
&=&-\na_Y R(Z,X,T)-\na_Z R(X,Y,T)\\
   &=&-\na_Y( R(Z,X)T)+ R(\na_YZ,X)T+ R(Z,\na_YX)T+ R(Z,X)\na_YT\\&&-\na_Z( R(Y,X)T)+ R(\na_ZY,X)T+ R(Y,\na_ZX)T+ R(Y,X)\na_ZT\\
   &=& R(\na_YZ,X)T+ R(Z,\na_YX)T+ R(\na_ZY,X)T+ R(Y,\na_ZX)T.
	\end{eqnarray*}
	
	According the  proposition.\ref{pr1-3-1}, we take: $R(V,V)(V_i)\subset V_i$  and   $\na_X R(Y, Z,T)\in
V_i$.

The otherwise,	
	
         \begin{eqnarray*} \na_X R(Y,Z,T)&=&\na_X( R(Y,Z)T)- R(\na_XY,Z)T- R(Y,\na_XZ)T- R(Y,Z)\na_XT\\
				&=&\na_X( R(Y,Z)T)- R(\na_XY,Z)T- R(Y,\na_XZ)T
\\&&+ R(Z,\na_XT)Y+ R(\na_XT,Y)Z. \end{eqnarray*}
				Then,  $\na_X(R(Y, Z)T)\in V_i$.

 Now, we will show that
 $\na_X\Ri(Y)\in V_i$.\\
 We choose  a pseudo-orthonormally basis  $(e_1,..., e_n)$  adapted  to the  decomposition (\ref{dec-Tm-ri}) and we put $\epsilon_k =\langle e_k,e_k\rangle$.
Let  $Z \in V_i^{\perp}$.

If  $e_k \in V_i$, we have seen that  $\na_X(R(Y, e_k)e_k)\in V_i$  and if  $e_k \in V_i^{\perp}$,
 we get   $ R(Y, e_k) = 0$. Then,
           \begin{eqnarray*}
           \langle \na_X(\Ri(Y)),Z\rangle&=& -\langle \Ri(Y), \na_XZ\rangle\\
            &=&\sum_{k=1}^n\epsilon_k \langle R(Y,e_k)e_k,\na_Xz\rangle\\&=&-\sum_{k=1}^n \epsilon_k \langle \na_X( R(Y,e_k)e_k),Z\rangle
   \\&=&0. \end{eqnarray*}
   So, $\na_X\Ri(Y)  \in V_i$.\\
 If $P_i(t)=t^2+at+b$ with $b\neq0$.  Then for all  $Y\in V_i$, we get  \[Y=-\frac{1}{b}(\Ri^2(Y)+a\Ri(Y))~~\text{ and},~~  \na_XY  \in V_i.\]
 If $P_i(t)=t-\la_i$ with $\la_i\neq0$.  Then, for all $Y\in V_i$, we get \[Y=\frac{1}{\la_i}\Ri(Y)~~\text{ and },~~ \na_XY  \in V_i.\]

So,  $\na_XV_i \subset V_i$, this shows that  $\na_{V_j}V_i \subset V_i$ and $\na_{V_0}V_i \subset V_i$,
 for all  $i,\; j  \geq1$ with  $i\neq  j$.

 The other results  are obtained immediatly   because the  metric $g$  is parallel ($\nabla g=0$).

    \end{proof}

    \begin{co}
     Let $(M,g)$ be  pseudo-Riemannian  semi-symmetric manifolds.  Let $\X=\prod_i P_i$ be the minimal  polynomial of $\Ri$. If we put $V_0=\ker(\Ri^2)$ and $\forall i\geq 1$, $V_i=\ker(P_i(\Ri))$. Then, for all $i\geq1$, the  distribution $V_0$  and $V_0+V_i$ are involutives.
    \end{co}

    \begin{rem}
    The distribution $V_0$ and $V_0+V_i$ are involutives spaces  and  not  necessairly parallels.
   \end{rem} 
%%%%%%%%%%%%%%%%%%%%%%%%%%%%%%%%%%%%%%%%%%%%
%¨%%%%%%%%%%%%%%%%%%%%%%%%%%%%%%%%%%%%%%%%%%%%%%%%%%%
%-------------------------------------------------------------------
%%%%%%%%%%%%%%%%%%%%%%%%%%%%%%%%%%%%%%%%%%%%%%%%%%%%%%%%%%%%%%%%%%%%
\section{Four dimensional semi-symmetric neutral Lie groups } \label{section4bis}

%%%%%%%%%%%%%%%%%%%%%%%%%%%%%%%%%%%%%%%%%%%%%%%%%%%%%%%%%%%%%
%%%%%%%%%%%%%%%%%%%%%%%%%%%%%%%%%%%%%%%%%%%%%%%%%%%%%%%%%ùù

In this section, we give some general properties of semi-symmetric neutral Lie groups and we prove Theorem \ref{main} when $M$ is a neutral Lie group.

A Lie group $G$ together with a left-invariant pseudo-Riemannian metric $g$ is called a 
\emph{pseudo-Riemannian Lie group}. The  metric $g$ 
defines a  pseudo-Euclidean product $\prs$ on the Lie algebra $\G=T_eG$ of $G$, and conversely, any  pseudo-Euclidean product on $\G$
gives rise
to an unique  left-invariant pseudo-Riemannian metric on $G$.\\ We will refer to a Lie
algebra endowed with a  pseudo-Euclidean product as a \emph{pseudo-Euclidean Lie algebra}.  The
Levi-Civita connection of $(G,g)$ defines a product $\mathrm{L}:\G\times\G\too\G$ called the Levi-Civita product and given by  Koszul's
formula:
\begin{eqnarray}\label{levicivita}2\langle
\mathrm{L}_uv,w\rangle&=&\langle[u,v],w\rangle+\langle[w,u],v\rangle+
\langle[w,v],u\rangle.\end{eqnarray}
For any $u,v\in\G$, $\mathrm{L}_{u}:\G\too\G$ is skew-symmetric and $[u,v]=\mathrm{L}_{u}v-\mathrm{L}_{v}u$. We will also write $u.v=\mathrm{L}_{u}v$.
The Riemannian curvature on $\G$ is given by:
%%%%%%%%%%%%%%%%%%%%%%%%%%%%%%%%%%%%%%%%%%%%%%%%%%%
\begin{equation}\label{curvature}
\Ku(u,v)=\mathrm{L}_{[u,v]}-[\mathrm{L}_{u},\mathrm{L}_{v}].
\end{equation}
It is well-known that $\Ku$ is a curvature tensor on $(\G,\prs)$ and, moreover, it satisfies the differential Bianchi identity
\begin{equation}\label{2bianchi}
\mathrm{L}_u(\Ku)(v,w)+\mathrm{L}_v(\Ku)(w,u)+\mathrm{L}_w(\Ku)(u,v)=0,\quad u,v,w\in\G
\end{equation}where $\mathrm{L}_u(\Ku)(v,w)=[\mathrm{L}_u,\Ku(v,w)]-\Ku(\mathrm{L}_uv,w)-\Ku(v,\mathrm{L}_uw).$
Denote by $\mathfrak{h}(\G)$ the holonomy Lie algebra of $(G,g)$. It is the smallest Lie algebra containing  $\mathfrak{h}(\Ku)=\mathrm{span}\{\Ku(u,v):u,v\in\G \}$ and satisfying $[\mathrm{L}_u,\mathfrak{h}(\G)]\subset\mathfrak{h}(\G)$, for any $u\in\G$.\\

$(G,g)$ is semi-symmetric iff $\Ku$ is a semi-symmetric curvature tensor of $(\G,\prs)$.  Without reference to any Lie group, we call a pseudo-Euclidean Lie algebra $(\G,\prs)$ semi-symmetric  if its curvature is semi-symmetric.

Let $(\G,\prs)$ be a semi-symmetric Lie algebra with metric $\prs$ of signature $(2,n)$ such that $n\geqslant2$. According to Proposition \ref{pr2}, $\G$ splits orthogonally as
\begin{equation}\label{split}\G=\G_0\oplus \G_1\oplus\ldots\oplus \G_r,
\end{equation}where $\G_0=\ker(\mathrm{Ric}^2)$ and $\G_1,\ldots,\G_r$ are the eigenspaces associated to the  non zero eigenvalues of $\mathrm{Ric}$,

 or 
\begin{equation}\label{split2}\G=\G_0\oplus \G_c\oplus \G_1\oplus\ldots\oplus \G_r,
\end{equation}where $\G_0=\ker(\mathrm{Ric})$,  $\G_c=\ker(\mathrm{Ric}^2-(z+\bar{z})\mathrm{Ric}+\mid z\mid^2 I)$ such that  $z$ is  non  real eigenvalue of $\mathrm{Ric}$ and $\G_1,\ldots,\G_r$ are the eigenspaces associated to the  non zero real eigenvalues of $\mathrm{Ric}$.\\
Moreover, $\Ku(\G_i,\G_j)=0$ for any $i\not=j$, $\dim(\G_c)=4$ and $\dim\G_i\geq2$ if $i\not=0$. According the Proposition\ref{pr121}, the following proposition gives  more properties of the $\G_i$'s involving the Levi-Civita product.

%%%%%%%%%%%%%%%%%%%%%%%%%%%%%%%%%%%%%%%%%%%%

\begin{pr}\label{pr7} Let $(\G,\prs)$ be a semi-symmetric Lie algebra  with metric $\prs$ of signature $(2,n)$ such that $n\geqslant2$. Then, for any $i,j\in\{c,1,\ldots,r\}$ and $i\not=j$,
	\begin{equation*}\label{l}
	{\G_j}.\G_i\subset \G_i, \;  {\G_i}.\G_i\subset \G_0+\G_i, \;
	{\G_0}.\G_i\subset \G_i, \;
	{\G_0}.\G_0\subset \G_0, \;  {\G_i}.\G_0\subset \G_0+\G_i.
	\end{equation*}
\end{pr}

Let $(G,g)$ be a four dimensional semi-symmetric neutral Lie group with Ricci curvature having a non zero eigenvalue
and,   according to \eqref{split} , \eqref{split2} and Proposition \ref{pr7}, the Lie algebra $\G$ of $G$ has one of the following types:
\begin{enumerate} \item[]$(S4\la)$: $\dim\G=4$ and $\G=\G_\la$ with $\la\not=0$.
	\item[]$(S4\mu\la)$:    $\G=\G_\mu\oplus\G_\la$ with $\dim\G_\mu=\dim\G_\la=2$,  $\la\not=\mu$,  $\la\not=0$, $\mu\not=0$, $\G_\mu.\G_\la\subset\G_\la$, $\G_\la.\G_\mu\subset\G_\mu$, $\G_\la.\G_\la\subset\G_\la$ and $\G_\mu.\G_\mu\subset\G_\mu$.
	\item[]$(S4\la0^1)$ : $\G=\G_0\oplus\G_\la$ with $\dim\G_0=1$, $\G_0.\G_\la\subset\G_\la$, $\G_0.\G_0\subset\G_0$ and $\la\not=0$.
	\item[]$(S4\la0^2)$:   $\G=\G_0\oplus\G_\la$ with $\dim\G_\la=2$, $\G_0.\G_\la\subset\G_\la$,  $\G_0.\G_0\subset\G_0$ and $\la\not=0$,
	\item[]$(Sz\bar{z})$: $\G=\ker(\mathrm{Ric}^2-(z+\bar{z})\mathrm{Ric}+\mid z\mid^2 I)$ with $z\in\C-\R$.
\end{enumerate}where  $\G_\la=\ker(\Ri-\la \mathrm{Id}_\G)$ and $\G_0=\ker(\Ri)$. 

%%%%%%%%%%%%%%%%%%%%%%%%%%%%%%%%%%%%%%%%%%%%%

   \begin{pr}\label{pr45}: \textbf{Ricci is neither flat nor isotropic}\\  Let $(\G,\prs)$ be a four dimensional semi-symmetric neutral Lie algebra with Ricci curvature  admitting a non-zero  eigenvalue. Then $\G$ is a symmetric space. Precisely, one of the following cases occurs:
     \begin{enumerate} \item $\G$ is of type $(S4\mu\la)$.   Then $\G_\la.\G_\mu=\G_\mu.\G_\la=0$  and   $\G$ is a product  of two  Lie  algebras  with the same  metric et and the same dimension $2$.
 \item $\G$  is of type $(S4\la0^1)$.  Then $\G.\G_0=0$, $\G_\la.\G_\la\subset\G_\la$ and hence $\G$ is the semi-direct product of  $\G_0$  with the three dimensional Lorentzian Lie algebra $\G_\la$ of constant curvature and the action of $\G_0$ on $\G_\la$ is by a skew-symmetric derivation.
    \item $\G$  is of type $(S4\la0^2)$.  Then  $\G_0.\G=0$, $\G_\la.\G_\la\subset\G_\la$, $\G_\la.\G_0\subset\G_0$ and hence $\G$ is the semi-direct product of the pseudo-Euclidean Lie algebra $\G_\la$  with the abelian Lie algebra $\G_0$ and the action of $\G_\la$ on $\G_0$ is given by skew-symmetric endomorphisms.
		\item $\G$ is of type $(S4\la)$  with  $\la\neq 0 $. In this case, we get   $\dim(\h( \Ku))\in\{2,4,6\}$.
		\item $\G$ is of type $(S4z\bar{z})$. In this case,  $\dim(\h( \Ku))=2$.
        \end{enumerate}
        \end{pr}
 \begin{proof}$ $

 -For types $(S4\mu\la)$, $(S4\la0^1)$ and $(S4\la0^2)$,  $\Ri$  admits  two real eigenvalues. Then each  eigenspace   is either Lorentzian or Riemannian and   the demonstration of similar cases in \cite{benromane}  remains valid  in  the current situation.

-For type $(S4\la)$, it is a result of the theorem(\ref{theo2}).

-For type $(S4z\bar{z})$, it is the same proof in the case Ricci complex for a  homogeneous semi-symmetric manifolds.
 \end{proof}
       %%%%%%%%%%%%%%%%%%%%%%%%%%%%%%%%%%%%%%%%%%%%%%%%%%%%%%%%%%%
       %%%%%%%%%%%%%%%%%%%%%%%%%%%%%%%%%%%%%%%%%%%%%%%%%%%%%%%%%%%

%%%%%%%%%%%%%%%%%%%%%%%%%%%%%%%%%%%%%%%%%%%%%%%%%%%%%%%%%%%%%%%%%%%%%
%%%%%%%%%%%%%%%%%%%%%%%%%%%%%%%%%%%%%%%%%%%%%%%%%%%%%%%%%%%%%%%%%%%%%

According to this proposition, we get the following theorem:
 \begin{theo}
 Let $G$ be a four-dimensional connexe   simply connected  neutral Lie group. If   
    $(G,g)$ is  semi-symmetric space  admitting a left invariant metric $g$ and  it's Ricci cuvature  admits   no zero eigenvalue, then  $G$ is   localy symmetric.
 \end{theo}
\begin{co}Let $\G$ be a four-dimensional semi-symmetric  nonsymmetric neutral Lie algebra. Then, its Ricci operator  satisfies  the condition $$\Ri^2 = 0 $$
i.e; Ricci is plat or istrope. Precisely, Ricci has only $0$  eigenvalue.  
  \end{co}
 \begin{rem} There are some four dimensional neutral  semi-symmetric non-symmetric  Lie algebras   with
   Ricci plat:

 \textbf{ Example:} Let $\G=vect(x,y,z,t)$ be a Lie algebra equipped with a metric $\prs$ given by: $\langle x,z\rangle=\langle y,t\rangle
=1$ and the non zero  brackets are:\\   $[x,y]= Ax+Bt$, $[x,z]= 2Dx$, $[y,z]= Cx-Dy+Az$, $[y,t]= -2At \esp [z,t]=-\frac{AD}{B}x-Dt.$\\
   Then $\G$ is semi-symmetric non symmetric   with  Ricci plat  and the   courvature:  $$ R= 4AC. A_{x,t}\vee A_{x,t} .$$

This example makes  the difference between the Lorentzian case and the case of the signature  $ (2,2) $: In the first case, the semi-symmetric Lie algebras of Ricci flat are flat and locally symmetrical. \end{rem}

%%%%%%%%%%%%%%%%%%%%%%%%%%%%%%%%%%%%%%%%%%%%%%%%%%%
\begin{rem} \end{rem} let $\G$ be a four-dimensional semi-symmetric   neutral Lie algebra with isotropic  Then, its Ricci. Then, two cases are   possibles: $rank(\Ri)=1$ or $rank(\Ri)=2$ and     $\dim(\h(  \Ku))\in\{1,2\} $  
and we get  the following Proposition:
%%%%%%%%%%%%%%%%%%%%%%%%%%%%%%%%%%%%%%%%%%%%%%%%%%%%%%%%%%%%%%%%%%
  \begin{pr}\label{isotp1}
  Let $(\G,\prs)$ be a four dimensional semi-symmetric Lie algebra with  $ \Ku$ the curvature tensor and    $\Ri$    the Ricci operator. 
	    \begin{enumerate}
\item  If $\Ku\neq 0$ and $\Ri=0$, then,    there is a basis  $(x,y,z,t)$ such that  $\langle x, t \rangle=\langle y,z\rangle=1 $ and   $ \Ku=a A_{x,z}\vee A_{x,y}$, where $a\in\R^*$.
	\item If $\Ri\neq \Ri^2=0$.  Then, one of the following situations is checked:

    \begin{enumerate}
  \item  $\mathfrak{h}( \Ku)$ is of  type $1.4^1$.  Then, there is a basis  $(x,y,z,t)$ such that  $\langle x, z \rangle=-\langle y , y \rangle=\langle t , t \rangle=1 $
                         and   $  \Ku=q A_{x,y}\vee A_{x,y}$, $\Ri= -q(x\vee x )$, $q\neq0$.  
 \item  $\mathfrak{h}( \Ku)$  is of type  $2.5^1$.   Then, there is a basis  $(x,y,z,t)$ such that  $\langle x, z \rangle=\langle y , t \rangle=1 $ and   \\$ \Ku=r A_{x,y}\vee A_{x,y}+ p.A_{x,t}\vee A_{x,t}+ q.A_{x,y}\vee A_{x,t}$,  $~~\Ri= q(x\vee x )$, $p\neq0\neq q$ and $r\neq0$.
  \item $\mathfrak{h}( \Ku)$  is of type  $2.2^1$.   Then, $\dim\mathfrak{h}(\Ku)=2$  and there is a basis $(x,y,z,t)$ such that   $\langle x, z \rangle=\langle y , t \rangle=1 $ and \\ $\Ku=s A_{x,t}\vee A_{x,t}+ p.((A_{x,z}+ A_{y,t})\vee .A_{x,t})$,  $\Ri= - p(x\vee t)$, $p\neq0\neq s$.
		%%%%%%%%%%%%%%%%%%%%%%%%%%%%%%%%%%%%

          \item $ \mathfrak{h}( \Ku)$  is of type  $2.2^2$.   Then,     there is a basis  $(x,y,z,t)$such that   $\langle x, t \rangle=\langle y , z \rangle=1 $ and  \\$ \Ku=p A_{x,y}\vee A_{x,y}+ q.A_{x,y}\vee (A_{x,z}+A_{y,t})$  et $\Ri= q(x\vee x - y\vee y)$ , $p\neq0\neq q$.
     \end{enumerate}
     \end{enumerate}
         \end{pr}

 In \cite{ali},   A. Haji-Badali and A. Zaeim give a complet classification of 
 four-dimensional   semi-symmetric nonsymmetric  neutral Lie algebras.

 %%%%%%%%%%%%%%%%%%%%%%%%%%%%%%%%%%%%%%%%%%%
%%%%%%%%%%%%%%%%%%%%%%%%%%%%%%%%%%%%%%%%%%%%
%%%%%%%%%%%%%%%%%%%%%%%%%%%%%%%%%%%%%%%%%%%%%

    \section{Proof of Theorem \ref{theo2} and \ref{main}}   \label{section3}

       Let $(M,g)$ be a four dimensional neutral  manifold and the tensor curvature $ R $ is considered a symmetric endomorphism in the space  $\Lambda^2 TM$; 
\begin{equation}\label{courburehomo}
\begin{array}{cccc}
 R : & \Lambda^2 TM & \rightarrow & \Lambda^2 TM \\
   & x\we y & \mapsto& {R}(x\we y):=R(x,y).
\end{array}
  \end{equation}
  Let  $J:~~ \Lambda^2TM  \rightarrow  \Lambda^2 TM $ be  a \texttt{ Hodge}  morphism given by:  \[\al \we \be= \langle J \al, \be \rangle_1 \om,\] for all  $m\in M$, $\al$, $\be \in \Lambda^2 T_mM$, $\om=e_1\we e_2\we e_3\we e_4$, such that  $(e_1,e_2,e_3,e_4)$ is a positive-oriented orthonormal basis of  $T_mM$  and  $\prs_1$ is the metric of  $\Lambda^2 T_mM$ induced by  $g$. It's easy to proof that $J^2=id_{\Lambda^2 TM}$ and we put  $\Lambda^+ T_mM$ and $\Lambda^- T_mM$ the eigenspaces of $J_m$ associated respectively to eigenvalues $1$ and  $-1$, they are the same dimension $3$.

   On the other hand, if   $R$ is   the Einstein curvature, we get,
   $$J\circ R=R\circ J.$$
  Therefore, $\Lambda^+ TM$ and $\Lambda^- TM$ are  invariant  by $R$.

     Moreover,
  \[e_1\we e_2\pm e_3\we e_4,~e_1\we e_3\mp e_2\we e_4,~~\text{and}~~e_1\we e_4\pm e_2\we e_3\] is a basis of  $\Lambda^\pm T_mM$.
    The proof is based on Corollary \ref{co4-21}   and the following   theorem proved in \cite{derd} 
    
    \begin{theo}\label{derd}\cite{derd} Let the self-dual curvature operator ${R^+} :  \Lambda^+ TM\rightarrow \Lambda^+ TM $   of
an oriented four-dimensional Einstein manifold $(M, g)$ of the metric signature
$(2,2)$  be complex-diagonalizable at every point, with complex eigenvalues forming
constant functions $M \longrightarrow \C$. If  $\nabla R^+ \neq 0$  somewhere in $M$, then $(M, g)$ is locally homogeneous, namely, locally isometric to a Lie group with a left-invariant metric.
More precisely, $(M, g)$ then is locally isometric to one of Petrov's Ricci-flat
manifolds,	
    \end{theo}

    \paragraph{Proof of Theorem\ref{theo2}}
       $   $
   
     Let $(M,g)$ be a four dimensional Einstein  semi-symmetric neutral manifold.  Then  the Ricci tensor  satisfies the following relationship:
 \begin{equation}
  \mathfrak{ ric}=\frac{\mathfrak{s}}{4}g,
 \end{equation} such that $\mathfrak{s}$ is the scalar curvature.
 
On the other way, the Ricci tensor  satisfies the following relationship:

            \begin{equation}\label{besse(ds)}
                         \delta(\mathfrak{ric})=-\frac{1}{2}d (\mathfrak{s}),
                           \end{equation}
                           where  $\delta$ and $d$  are the contravariant and the covariant differential on $M$  respectively  (See \cite{besse},Proposition1.94, page 43).\\
                          This induces that the scalar curvature  $\mathfrak{s}$  is a constant   function.

 So, $(M,g)$ is  semi-symmetric  and according of the  corollary\ref{co4-21} we get that  morphism  $R^+ :  \Lambda^+ TM  \too \Lambda^+ TM $ is  diagonalizable and it satisfies of one of  the following  situations:\begin{enumerate}
                                                              \item ${R^+}$  is  an homothety with a report  $\frac{\mathfrak{s}}{4}$,
                                                              \item ${R^+}$  is diagonalizable as $\C$-linear endomorphism of $\wedge^+ T_pM$ with eigenvalues 0 and $-\frac\la4$  of multiplicity $2$ and $1$ respectively,  where $\la$ is the scalar curvature.  
                                                            \end{enumerate}
    Then $M$ admits a   no  flat Ricci.  According to theorem\ref{derd},  we get  that $M$ is localy  symmetric.  This completes the proof of the theorem.\ref{theo2}.\\

%%%%%%%%%%%%%%%%%%%%%%%%%%%%%%%%%%%%%%%%%%%%%%%%%%%%%%%%%%%%%%%
%%%%%%%%%%%%%%%%%%%%%%%%%%%%%%%%%%%%%%%%%%%%%%%%%%%%%%%%%%%%%%%%%%
%%%%%%%%%%%%%%%%%%%%%%%%%%%%%%%%%%%%%%%%%%%%%%%%%%%%%%%%%%%%%%%%%

    \paragraph{Proof of Theorem \ref{main}} 
    $   $
    
\begin{proof}
Let $(M,g)$ be a four dimensional simply connected homogeneous  neutral  semi-symmetric manifold with Ricci curvature having a non zero eigenvalue. According to  proposition\ref{pr43}. If the  Ricci curvature has a non zero eigenvalue in $\C$, Then, $M$ is one of the following   types: $(S4\la)$, $(S4\la\mu)$,   $(S4\la0^1)$,  $(S4\la0^2)$ or  $(Sz\bar{z})$.
 
 \begin{Le}\label{lemma61}
Let $(M,g)$ be a four dimensional simply connected homogeneous  neutral  semi-symmetric manifold of type $(S4\la\mu)$ or  $(S4\la0^1)$ or  $(S4\la0^2)$ then $(M,g)$ is either Ricci-parallel or locally isometric to a Lie group equipped with a left invariant neutral metric.
\end{Le}

 \textbf{Proof of Lemma \ref{lemma61}}

If  $M$ is one of the following types $(S4\la \mu)$ or  $(S4\la0^2 )$, the proof of  theorem 4.1  in  \cite{calvaruso3}   remains valid in the current situation. For the case  $(S4\la0^1)$  and according to \cite{komrakov}, we find the same homogeneous   manifolds as the theorem 4.6 in \cite{calvaruso3} and consequently its proof remains valid in the current situation. This completes the proof of the Lemma\ref{lemma61}.$\blacksquare$

 So if   $(M,g)$ is of type  $(S4\la)$ such that $\la\neq0$, that  is, $M$ is the Einstein space with non null scalar curvature  and we can apply Theorem \ref{theo2} to get that $M$ is locally symmetric.

 If $(M,g)$ is Ricci-parallel and the Ricci operator has two distinct real eigenvalues then, according to Theorem 7.3  \cite{boubel} and the Proposition\ref{pr121} , $(M,g)$ is a product of two Einstein homogeneous semi-symmetric pseudo-Riemannian manifolds of dimension less or equal 3 and according the some results of same situation in  \cite{benromane} we get that  $(M,g)$ is localy symmetric.
  
  If $M$ is  locally isometric to a Lie group  equipped with a left invariant neutral metric, we have shown in section \ref{section4bis} that $M$ is locally symmetric.

%%%%%%%%%%%%%%%%%%%%%%%%%%%%%%%%%%%%%%%%%%%%%%%%%%%%%%%%%%%
 Suppose now that   $(M,g)$ is of type $(Sz\bar{z})$.   Let $z=a+ib$ and   $\bar{z}=a-ib$  be the eigenvalues of   $\Ri$ such that  $b\neq0$. Then there is a  pseudo-orthonormal frame  $\B=(e,f,u,v)$ such that  $g(e,u)=g(f,v)=1$ in which the matrix of  $\Ri$ has the following form:  
\[ [\Ri]_{\B}=\left(
    		\begin{array}{cccc}
    		a& -b &0 & 0 \\
    		b & a & 0 & 0 \\
    		0 & 0 & a & -b \\
    		0 & 0 & b &a \\
    		\end{array}
    		\right)\]
Let $J$ be the  operator given by $J:=\frac{1}{b} (\Ri-a.I)$.  We have  $J^2=- I$ such that $I$ is the identity  operator of the tangent fibre of $M$. Then, by complexification, we get that the complex semi-symmetric manifolds $(M^{\mathbb{C}},g^{\mathbb{C}})$  which   $\Ri^{\mathbb{C}} $ the complex Ricci  is semi-symmetric operator admitting two eigenvalues  $z=a+ib$ and    $\bar{z}=a-ib$ which are the  constant functions because $M$ is homogeneous. Moreover, $\Ri^{\mathbb{C}} $    must be diagonalizable in  ${\mathbb{C}}$.  More precisely, $\ker(\Ri^{\mathbb{C}} -z.I)=\ker(J -i.I)$ and  $\ker(\Ri^{\mathbb{C}} -\bar{z}.I)=\ker(J +i.I)$.  Applying the procedure of the proof of the proposition\ref{pr121}, we find that the two two-dimensional  orthogonaly eigenspaces of $ \Ri^{\mathbb {C}} $ are parallel and consequently, they are locally symmetric. So $ (M ^ {\mathbb {C}}, g ^ {\mathbb {C}}) $ is locally symmetric. As a result, $ (M, g) $ is locally symmetric. This completes the proof of the Teorem\ref{main}.
\end{proof}

%%%%%%%%%%%%%%%%%%%%%%%%%%%%%%%%%%%%%%%%%%%%%%%%%%%%%%%%%%%
%%%%%%%%%%%%%%%%%%%%%%%%%%%%%%%%%%%%%%%%%%%%%%%%%%%%%%%%%%%    
    
    \section{Four-dimensional Ricci flat and   Ricci isotropic homogeneous semi-symmetric neutral manifolds } \label{section7}
    
    In this section, we deal with non flat semi-symmetric  four-dimensional neutral manifolds with ithe Ricci curvature is either isotropic  or flat . 
    
    We use Komrakov's classification \cite{komrakov} of four-dimensional homogeneous pseudo-Riemannian manifolds and we apply the following algorithm to find among Komrakov's list the pairs $(\overline{\G},\G)$ corresponding to    four-dimensional Ricci flat or   Ricci isotropic homogeneous semi-symmetric neutral manifolds which are not locally symmetric.

    Let $M=\overline{G}/G$ be an homogeneous manifold with $G$ connected and $\overline{\G}=\G\oplus\mathfrak{m}$, where $\overline{\G}$ is the Lie algebra of $\overline{G}$, $\G$ the Lie algebra of $G$ and $\mathfrak{m}$  an arbitrary complementary of $\G$ (not necessary $\G$-invariant). The pair $(\overline{\G},\G)$ uniquely defines the isotropy representation $\rho:\G\too\mathrm{gl}(\mathfrak{m})$ by $\rho(x)(y)=[x,y]_\mathfrak{m}$, for all $x\in\G$, $y\in\mathfrak{m}$. Let $\{e_1,\ldots,e_r,u_1,\ldots,u_n \}$ be a basis of $\overline{\G}$ where $\{e_i\}$ and $\{u_j \}$ are bases of $\G$ and $\mathfrak{m}$, respectively. The algorithm goes as follows.
    
    \begin{enumerate}\item Determination of invariant pseudo-Riemannian metrics on $M$:
		
		It is well-known that invariant pseudo-Riemannian metrics on $M$ are in a one-to-one correspondence with nondegenerate invariant symmetric bilinear forms on $\mathfrak{m}$. A symmetric bilinear form on $\mathfrak{m}$ is determined by its matrix $B$ in $\{u_i\}$ and its invariant if $\rho(e_i)^t\circ B+B\circ\rho(e_i)=0$ for $i=1,\ldots,r$.
    	\item Determination of the Levi-Civita connection: 
			
			Let $B$ be a nondegenerate invariant symmetric bilinear form on $\mathfrak{m}$. It defines uniquely an invariant linear Levi-Civita connection $\na:\bar{\G}\too\mathrm{gl}(\mathfrak{m})$ given by
    	\[ \na(x)=\rho(x),\;\na(y)(z)=\frac12[y,z]_\mathfrak{m}+\nu(y,z),\; x\in\G, y,z\in\mathfrak{m}, \] where $\nu:\mathfrak{m}\times\mathfrak{m}\too\mathfrak{m}$ is given by the formula
    	\[ 2B(\nu(a,b),c)=B([c,a]_\mathfrak{m},b)+B([c,b]_\mathfrak{m},a),\;a,b,c\in\mathfrak{m}. \]
    	\item Determination of the curvature: 
			
			The curvature of $B$ is the bilinear map
    	$\Ku:\mathfrak{m}\times\mathfrak{m}\too\mathrm{gl}(\mathfrak{m})$ given by
    	\[ \Ku(a,b)=[\na(a),\na(b)]-\na([a,b]_\mathfrak{m})-\rho([a,b]_{\G}),\; a,b\in\mathfrak{m}. \]
    	\item Determination of the Ricci curvature:
			
			It is given by its matrix in $\{ u_i\}$, i.e., $\ric=(\ric_{ij})_{1\leq i,j\leq n}$ where
    	\[ \ric_{ij}=\sum_{r=1}^n\Ku_{ri}(u_r,u_j). \]
    	\item Determination of the Ricci operator: 
			
			We have $\Ri=B^{-1}\ric$.
    	\item Checking  the semi-symmetry condition. 
    	
    \end{enumerate} 
    The following theorem gives the list of four dimensional homogeneous  neutral  semi-symmetric  manifolds non flat  which Ricci is either isotropic or flat. 
    
    %%%%%%%%%%%%%%
   \begin{theo}\label{main3} Let  $M=\bar{G}/G$  be a four dimensional  homogeneous neutral  semi-symmetrique no symmetric manifolds   which Ricci is either isotropic or flat.\\  Let $\bar{\G}=span\{e_1,..,e_n,u_1,u_2,u_3,u_4\}$ and  $\G=span\{e_1,..,e_n\}$ the  Lie  algebras   associted respectively to  $\bar{G}$ and  $G$.  Then,  $M$ is isometric to one of the following  types:
   %%%%%%%%%%%%%%%%%%%%%%%%%%%%%%%%%%%%%%%%%%%%%%%%%%%%%%%%%%%
   \begin{enumerate}\item[I)] $\bar{\G}=span\{e_1, u_1,u_2,u_3,u_4\}$;   \begin{enumerate}  
   \item  $\langle u_1,u_4\rangle=-\langle u_2,u_3\rangle=a$,    $\langle u_3,u_3\rangle=-\langle u_4,u_4\rangle=b$,       $\langle u_3,u_4\rangle=c$\\ 
  ${\mathbf{1.3}^1}:2,3,4,6, 7,10,15,16, 24,26-30$,\\
   ${\mathbf{1.3}^1}:5 $ with  $(\la,\mu)\neq(0,2),$\\
  ${\mathbf{1.3}^1}:8,19,20,22$ with  $b\neq0$,\\
  ${\mathbf{1.3}^1}:9 $ with  $b\la(\la+1)\neq0,$\\
   ${\mathbf{1.3}^1}:12 $ with  $(\la-\mu-1)(\la-\mu+1)\neq0,$\\
  ${\mathbf{1.3}^1}:13 $ with  $\la\neq\frac{1}{2}$ ,\\
    $ {\mathbf{1.3}^1}: 21  $ with  $b\la(\la-1)\neq0,$\\
   ${\mathbf{1.3}^1}:25 $ with  $(b,\la)\neq(0,2),$\\
   \item  ${\mathbf{1.4}^1}:$    with  $\langle u_1,u_3\rangle=-\langle u_2,u_2\rangle=a,~~ \langle u_3,u_3\rangle=b,~\langle u_3,u_4\rangle=d,~\langle u_4,u_4\rangle=a-b$,\\ 
         ${\mathbf{1.4}^1}:2$    with  $b\neq0$ and $p=1$,\\
  ${\mathbf{1.4}^1}:9-11,13,15-20$.
\end{enumerate}
\item[II)]  $\bar{\G}=span\{e_1,e_2,u_1,u_2,u_3,u_4\}$ with  $\langle u_1,u_3\rangle=\langle u_2,u_4\rangle=a$;\begin{enumerate}
%%%%%%%%%%%%%%%%%%%%%%%%%%%%%%%%%%%%%%%%%%%%%%%%%%%%%%%%%%%%%%%%
\item    
             ${\mathbf{2.2}^1}:2$ with $\la(\la^2-4)\neq0$  
             \item  ${\mathbf{2.2}^1}:3$.  
%%%%%%%%%%%%%%%%%%%%%%%%%%%%%%%%%%%%%%%%%%% 
          \item
            ${\mathbf{2.5}^1}:3-6$.\\
%%%%%%%%%%%%%%%%%%%%%%%%%%%%%%%%%%%%%%
 \end{enumerate}
\item[III)]  $\bar{\G}=span\{e_1,e_2,e_3,u_1,u_2,u_3,u_4\}$; \\
%%%%%%%%%%%%%%%%%%%%%%%%%%%%%%%%%
   ${\mathbf{3.3}^1}:1$  with $\langle u_1,u_3\rangle=\langle u_2,u_4\rangle=a$: \\
\end{enumerate}
\end{theo}

%%%%%%%%%%%%%%%%%%%%%%%%%%%%%%%%%%%%%%%%%%%%%%%%%%%%%%%%%%%%%%%%%%%%%%%%%%%%%%%%%%%%%%%%%%%%%%%%%%%%%%%%
%%%%%%%%%%%%%%%%%%%%%%%%%%%%%%%%%%%%%%%%%%%%%%%%%%%%%%%%%%%%%%%%%%%%%%%%%%%%%%%%%%%%%%

    %%%%%%%%%%%%%%%%%%%%%%%%%%%%%%%%%%%%%%%%%%%%%%%%%%%%%%%%%%%%%%%%%%%%%%%%%%%%%%%%%%%%%%%%%
    %%%%%%%%%%%%%%%%%%%%%%%%%%%%%%%%%%%%%%%%%%%%%%%%%%%%%%%%%%%%%%%%%%%%%%%%%%%%%%%%%%%%%%%%%


\begin{thebibliography}{99}
    	
  
\bibitem{benromane} A. Benroummane, M. Boucetta, A. Ikemakhen,  Four-dimensional homogeneous semi-symmetric Lorentzian manifolds,    Differential Geometry and its Applications 56 (2018) 211-233. 
   
\bibitem{besse}  Ar.L. Besse, Einstein manifolds, Classic in Mathematics Springer (2008).

\bibitem{boubel} C. Boubel, Sur l'holonomie des vari\'et\'es pseudo-Riemanniennes, Ph. D. these, Institut Elie Cartan de Nancy (3 mai 2000)

\bibitem{cartan} E. Cartan, Le\c con sur la g\'eom\'etrie des espaces de Riemann, 2nd. Edition, Paris, 1946.
    	 	
\bibitem{calvaruso3} Giovanni Calvaruso and Amirhesam Zaeim, Four-dimensional homogeneous neutral manifolds, Monatsh Math (2014) 174:377-402.
    	
\bibitem{derd} Andrzej Derdzinski, Curvature homogeneous indefinite Einstein metrics in dimension four: the diagonalizable case, Contemporary Mathematics Volume 337 (2003) 21-38.
    	
\bibitem{ali} A. Haji-Badali,  A. Zaeim, Semi-symmetric four dimensional neutral Lie groups. Czechoslovak Mathematical Journal, 70 (145) (2020), $393-410.$
    	
\bibitem{komrakov} Komrakov Jnr, B., Einstein-Maxwell equation on four-dimensional homogeneous spaces.
    	Lobachevskii J. Math. 8, 33-165 (2001).
    	
\bibitem{zabo} Z. I. Szabo, Structure theorems on Riemannian manifolds satisfying $R(X, Y) . R = 0$, I, the local version, J. Differential Geom. 17 (1982), 531-582.
    	
\bibitem{zabo1} Z. I. Szabo, Structure theorems on Riemannian manifolds satisfying $R(X, Y) . R = 0$, II, global version,
    	Geometriae Dedicata 
    	Volume 19, Issue 1 , pp 65-108. 
    	
\bibitem{takagi} H. Takagi, An example of Riemannian manifold satisfying $R(X, Y ).R=0$ but not
    	$\na R = 0$, T\^ohoku Math. J. 24 (1972), 105-108.
    	
    	
    \end{thebibliography}
\end{document}